\documentclass[10pt]{amsart}
\usepackage{amssymb,amstext,amsmath,amscd,amsthm,amsfonts,enumerate,latexsym,stmaryrd,multicol,geometry,graphicx,mathrsfs,bm}
\usepackage[usenames]{color}
\usepackage[all]{xy}
\newtheorem{thm}{Theorem}[section]
\newtheorem{lem}[thm]{Lemma}
\newtheorem{prop}[thm]{Proposition}
\newtheorem{cor}[thm]{Corollary}
\theoremstyle{definition}
\newtheorem{dfn}[thm]{Definition}
\newtheorem{ques}[thm]{Question}

\newtheorem{rem}[thm]{Remark}

\newtheorem{ex}[thm]{Example}

\theoremstyle{remark}

\newtheorem*{claim*}{Claim}
\newtheorem*{ac}{Acknowledgments}
\newtheorem*{conv}{Convention}

\numberwithin{equation}{thm}
\def\A{\mathcal{A}}
\def\add{\operatorname{\mathsf{add}}}
\def\Ass{\operatorname{Ass}}

\def\c{\operatorname{\mathsf{C}}}
\def\C{\mathcal{C}}

\def\CM{\mathsf{CM}}

\def\D{\mathcal{D}}
\def\d{\operatorname{\mathscr{D}}}
\def\db{\operatorname{\mathsf{D^b}}}
\def\depth{\operatorname{depth}}
\def\dim{\operatorname{dim}}
\def\ds{\operatorname{\mathsf{D_{sg}}}}

\def\edim{\operatorname{edim}}

\def\Ext{\operatorname{Ext}}

\def\ge{\geqslant}
\def\H{\mathrm{H}}
\def\Hom{\operatorname{Hom}}

\def\inf{\operatorname{inf}}

\def\le{\leqslant}

\def\m{\mathfrak{m}}
\def\max{\operatorname{max}}
\def\Max{\operatorname{Max}}
\def\Min{\operatorname{Min}}
\def\mod{\operatorname{\mathsf{mod}}}
\def\n{\mathfrak{n}}
\def\NF{\operatorname{NF}}

\def\p{\mathfrak{p}}

\def\proj{\operatorname{\mathsf{proj}}}
\def\q{\mathfrak{q}}

\def\rank{\operatorname{rank}}

\def\S{\mathcal{S}}
\def\Sing{\operatorname{Sing}}
\def\size{\operatorname{size}}
\def\soc{\operatorname{soc}}
\def\Spec{\operatorname{Spec}}
\def\sup{\operatorname{sup}}
\def\syz{\Omega}
\def\T{\mathcal{T}}

\def\thick{\operatorname{\mathsf{thick}}}

\def\V{\mathrm{V}}
\def\X{\mathcal{X}}
\def\Y{\mathcal{Y}}

\def\ZZ{\mathbb{Z}}
\begin{document}
\title{Ranks of Verdier quotients of derived categories over local rings} 
\author{Souvik Dey}
\address[Dey]{Department of Mathematical Sciences, University of Arkansas, 850 West Dickson Street, Fayetteville, Arkansas 72701, USA}
\email{souvikd@uark.edu}
\urladdr{https://sites.google.com/view/souvikdey/home}
\author{Yuki Mifune}
\address[Mifune]{Graduate School of Mathematics, Nagoya University, Furocho, Chikusaku, Nagoya 464-8602, Japan}
\email{yuki.mifune.c9@math.nagoya-u.ac.jp}
\begin{abstract}
Let $R$ be a commutative noetherian local ring with residue field $k$.
Denote by $\operatorname{\mathsf{D^b}}(R)$ the bounded derived category of finitely generated $R$-modules.
In this paper, we introduce the notion of ranks of triangulated categories and study Verdier quotients of $\operatorname{\mathsf{D^b}}(R)$ from the viewpoint of this invariant.
Our main result determines the rank of the Verdier quotient $\operatorname{\mathsf{D^b}}(R)/\operatorname{\mathsf{thick}}(R\oplus k)$ in terms of the ranks of the categories of maximal Cohen--Macaulay modules on the punctured spectrum.
\end{abstract}
\maketitle
\section{Introduction}
Let $R$ be a commutative noetherian local ring with residue field $k$.
We denote by $\mod R$ the category of finitely generated $R$-modules, and by $\db(R)=\db(\mod R)$ the bounded derived category of $\mod R$.
The Rouquier dimension of a triangulated category measures the number of extensions required to build the whole category out of a single object, and has been extensively studied; see \cite{AT2015,BV,DT2015a,DT2015b,Rou} for instance.
In this paper, we introduce the notion of the rank of a triangulated category.
Roughly speaking, this invariant measures the number of extensions required to generate the whole category from a single object, up to taking finite direct sums and direct summands, but without allowing shifts; for the precise definition, see Definition \ref{def_rank}.
The main difference between rank and Rouquier dimension is that rank does not allow shifts in the generation process.
In particular, rank is always at least the Rouquier dimension.
For example, if $R$ is excellent, then a theorem of Aoki \cite{Aoki} shows that $\db(R)$ has finite Rouquier dimension, whereas the rank of $\db(R)$ is always infinite; see Remark \ref{rank_db}.
Thus, rank detects information that is not captured by Rouquier dimension.
These observations motivate the study of the finiteness and precise values of ranks for Verdier quotients of $\db(R)$.
A fundamental Verdier quotient of $\db(R)$ is the singularity category, denoted by $\ds(R)=\db(R)/\thick(R)$.
When $R$ is Gorenstein, the rank of $\ds(R)$ coincides with the rank, in the sense of Dao--Takahashi \cite{DT2014}, of the category $\CM(R)$ of maximal Cohen--Macaulay $R$-modules; see Proposition \ref{frob}.
Another important example is obtained by further quotienting the singularity category by the thick closure of the residue field $k$. 
Namely, we consider the Verdier quotient
\[
\d(R)=\ds(R)/\thick(k)=\db(R)/\thick(R\oplus k).
\]
This category was studied in \cite{MT}, where the condition that $\d(R)$ admits an additive generator was characterized in terms of being locally of finite Cohen--Macaulay representation type on the punctured spectrum.
Note that the category $\d(R)$ is trivial if and only if $R$ has an isolated singularity by \cite{T2014}.
Since rank zero is equivalent to admitting an additive generator, our aim in this paper is to study the structure of $\d(R)$ in a broader framework by estimating its rank.
The main result of this paper is the following.
\begin{thm}[Corollary \ref{equality}]\label{thm_int}
Let $(R,\m)$ be a noetherian local ring which is locally Gorenstein on the punctured spectrum.
Assume that $R$ has finite singular locus and does not have an isolated singularity.
Then the following equality holds:
\[
\rank\d(R)=\sup\{\rank\CM(R_{\p})\mid\p\in\Sing(R)\setminus\{\m\}\}.
\]
\end{thm}
Here, $\Sing(R)$ denotes the singular locus of $R$, namely, the set of prime ideals $\p$ such that $R_{\p}$ is not regular.
As an application of Theorem \ref{thm_int}, we give examples in the cases where the rank of $\d(R)$ is precisely one and where it is precisely two.
These examples are not locally of finite Cohen--Macaulay representation type on the punctured spectrum, and hence the result of \cite{MT} does not apply to them.
Moreover, Theorem \ref{thm_int} suggests that the rank of $\d(R)$ can be detected from the localizations over the punctured spectrum.
Indeed, we give an example where $\d(R)$ is nontrivial and has finite rank even though $R$ itself is not excellent; see Example \ref{eg_notexc}.

The organization of this paper is as follows.
In Section 2, we introduce the ranks of triangulated categories and compare them with other invariants related to generation. 
We also study the finiteness of the rank of the singularity category.
In Section 3, we study the rank of the Verdier quotient $\d(R)$ in terms of its localizations on the punctured spectrum, prove the main theorem, and give several applications.
In Section 4, we show that the local Cohen--Macaulay assumption in \cite[Theorem 4.7]{MT} can be removed by using finite syzygy type.
\begin{conv}
Throughout this paper, $R$ denotes a commutative noetherian ring.
When $R$ is local, we denote its maximal ideal by $\m$ and its residue field by $k$, unless otherwise specified.
All subcategories are assumed to be strictly full.
\end{conv}
\section{Ranks of triangulated categories}
In this section, we introduce the notion of ranks of triangulated categories and compare it with other invariants related to generation. 
We then study the rank of the singularity category.
We begin by defining the ranks of triangulated categories and recalling the ranks and sizes of subcategories of abelian categories.
For the latter notion, see \cite{DT2014,IT2016}.
\begin{dfn}\label{def_rank}
Let $\T$ be a triangulated category.
Let $\A$ be an abelian category.
\begin{enumerate}[\rm(1)]
\item
Let $\C$ be a subcategory of $\T$.
Set ${|\C|}_{0}^{\T}=0$.
We denote by ${|\C|}_{1}^{\T}=\add_{\T}\C$ the additive closure of $\C$ in $\T$.
For subcategories $\C,\D$ of $\T$, we denote by $\C\ast\D$ the subcategory of $\T$ consisting of objects $E$ such that there exists an exact triangle
$C\to E\to D\to C[1]$
in $\T$ with $C\in\C$ and $D\in\D$.
For an integer $r>1$, we inductively define ${|\C|}_{r}^{\T}={\left|{|\C|}_{r-1}^{\T}\ast{|\C|}_{1}^{\T}\right|}_{1}^{\T}$.
The {\em rank} of $\T$, denoted by $\rank\T$, is defined as the infimum of integers $n\ge 0$ such that $\T={|G|}_{n+1}^{\T}$ for some object $G\in\T$.
Similarly, we set ${\langle\C\rangle}_{0}^{\T}=0$, ${\langle\C\rangle}_{1}^{\T}=\add_{\T}\{C[i]\mid C\in\C\text{ and }i\in\mathbb{Z}\}$, and ${\langle\C\rangle}_{r}^{\T}={\left\langle{\langle\C\rangle}_{r-1}^{\T}\ast{\langle\C\rangle}_{1}^{\T}\right\rangle}_{1}^{\T}$ for $r>1$.
The \em{Rouquier dimension} of $\T$, denoted by $\dim\T$, is defined as the infimum of integers $n\ge 0$ such that $\T={\langle G\rangle}_{n+1}^{\T}$ for some object $G\in\T$.
\item
Let $\C$ be a subcategory of $\A$.
Set ${|\C|}_{0}^{\A}=0$.
We denote by ${|\C|}_{1}^{\A}=\add_{\A}\C$ the additive closure of $\C$ in $\A$.
For subcategories $\C,\D$ of $\A$, we denote by $\C\ast\D$ the subcategory of $\A$ consisting of objects $E$ such that there exists an exact sequence
$0\to C\to E\to D\to 0$
in $\A$ with $C\in\C$ and $D\in\D$.
For an integer $r>1$, we inductively define ${|\C|}_{r}^{\A}={\left|{|\C|}_{r-1}^{\A}\ast{|\C|}_{1}^{\A}\right|}_{1}^{\A}$.
When $\A=\mod R$, we simply write ${|\X|}_{r}^{R}$ for ${|\X|}_{r}^{\mod R}$.
For a subcategory $\X$ of $\A$, the {\em size} of $\X$, denoted by $\size\X$, is defined as the infimum of integers $n\ge 0$ such that $\X\subseteq {|G|}_{n+1}^{\A}$ for some object $G\in\A$.
The {\em rank} of $\X$, denoted by $\rank\X$, is defined as the infimum of integers $n\ge 0$ such that $\X={|G|}_{n+1}^{\A}$ for some object $G\in\X$.
\end{enumerate}
\end{dfn}
\begin{rem}\label{rem_rank}
Let $\T$ be a triangulated category.
\begin{enumerate}[\rm(1)]
\item
An invariant closely related to the rank of a triangulated category is the Rouquier dimension \cite{Rou}.  
In the definition of the Rouquier dimension, one allows shifts at each step.
This is different from the rank defined above, where shifts are not included in the operation ${|-|}_{1}^{\T}$. 
In particular, one has $\dim \T\leq \rank \T$.
\item
The same definition of rank can also be made for exact categories. 
Indeed, if $\mathcal{E}$ is an exact category, then one defines $\C\ast\D$ for subcategories $\C,\D$ of $\mathcal{E}$ to be the subcategory consisting of objects $E$ admitting a conflation $C\to E\to D$ in $\mathcal{E}$ with $C\in\C$ and $D\in\D$, and then defines ${|\C|}_{r}^{\mathcal{E}}$ inductively as above.
With this definition, if $\X$ is a subcategory of an abelian category $\A$ which is closed under extensions and direct summands, and if $\X$ is regarded as an exact category with the exact structure induced from $\A$, then the rank of the exact category $\X$ coincides with the (Dao--Takahashi) rank of $\X$ as a subcategory of $\A$.
\end{enumerate}
\end{rem}
\begin{rem}\label{rem_basic_rank} 
Let $\T$ be a triangulated category and $\C,\D$ subcategories of $\T$. 
We record some elementary properties of the operation ${|-|}_{n}^{\T}$ that will be used repeatedly. 
By the associativity of $\ast$ and its compatibility with shifts and finite direct sums, for all positive integers $a,b$ and all $i\in\mathbb{Z}$, one has ${\left|{|\C|}_{a}^{\T}\ast{|\C|}_{b}^{\T}\right|}_{1}^{\T}\subseteq{|\C|}_{a+b}^{\T}$ and $({|\C|}_{b}^{\T})[i]={|\C[i]|}_{b}^{\T}$; see \cite[Lemma 2.3]{T2023c}. 
Consequently, if $\C\subseteq{|\D|}_{a}^{\T}$, then ${|\C|}_{b}^{\T}\subseteq{|\D|}_{ab}^{\T}$.
In particular, we have ${\left|{|\D|}_{a}^{\T}\right|}_{b}^{\T}\subseteq{|\D|}_{ab}^{\T}$. 
\end{rem} 
We begin with some elementary examples in which the rank can be observed explicitly.
\begin{ex}\label{eg_rank1}
\begin{enumerate}[\rm(1)]
\item
Let $(S,\n)$ be a discrete valuation ring with $\n=(x)$.
Consider the ring $R=S/(x^n)$, where $n\ge1$.
The structure theorem for finitely generated modules over a principal ideal domain implies that $\mod R=\add\{R/(x^i)\mid1\le i\le n\}$.
Hence the category $\mod R$ is of finite type, and we conclude that $\rank(\mod R)=0$.
\item
Let $(S,\n)$ be a two-dimensional regular local ring with $\n=(x,y)$.
Consider the ring $R=S/(x^2,y^n)$, where $n\ge2$.
Applying (1) to the discrete valuation ring $S/(x)$, we have $\mod R/(x) = \add_{R}\{R/(x,y^i)\mid1\le i\le n\}$.
For every $M\in\mod R$, there exists an exact sequence $0\to (0:_{M}x)\to M\to xM\to0$.
Hence we have $\mod R={|\mod R/(x)|}_{2}^{R}={|\bigoplus_{i=1}^{n}R/(x,y^i)|}_{2}^{R}$.
This implies that $\rank(\mod R)\le1$.
Since $R$ is not a principal ideal ring, the category $\mod R$ is not of finite type by \cite[Theorem 3.3]{LW}.
Therefore, one has $\rank(\mod R)=1$.
\item
Let $R$ be a singular artinian local ring.
By \cite[Lemma 3.2]{M8}, one has $\Omega\mod R\subseteq {|R\oplus k|}_{\ell\ell(R)-1}^{R}$, where $\ell\ell(R)=\inf\{n\geq1\mid \m^{n}=0\}$ denotes the Loewy length of $R$.
If, in addition, $R$ is Gorenstein, then $\mod R={|R\oplus k|}_{\ell\ell(R)-1}^{R}$.
This implies that $\rank(\mod R)\leq \ell\ell(R)-2$.
If, moreover, $R$ is a complete intersection, then by \cite[Corollary 5.10]{BIKO2010} one has $\edim R-1\leq\dim\ds(R)\leq\rank(\mod R)\leq\ell\ell(R)-2$.
Hence, if $\edim R=\ell\ell(R)-1$, then we obtain $\rank(\mod R)=\edim R-1=\ell\ell(R)-2$.

For example, consider the ring $R=S/(x_{1}^{2},\ldots,x_{n}^{2})$, where $S$ is an $n$-dimensional regular local ring with a regular system of parameters $x_{1},\ldots,x_{n}$.
Then $\edim R=n$ and $\ell\ell(R)=n+1$, so $\edim R=n=\ell\ell(R)-1$. 
Thus one has $\rank(\mod R)=n-1$.
\end{enumerate}
\end{ex}
We observe that $\db(R)$ does not have finite rank unless $R$ is the zero ring.
\begin{rem}\label{rank_db}
Let $R$ be a commutative noetherian ring. If $R\neq0$, then
one has $\db(R)\neq\bigcup_{n\geq 0}{|G|}_{n}^{\db(R)}$ for every $G\in\db(R)$. In particular, the ring $R$ is zero if and only if $\rank\db(R)$ is finite.
We prove the first assertion. For integers $a\le b$, set $\X(a,b)=\{X\in\db(R)\mid \H^{i}X=0\text{ for all } i<a \text{ and all } i>b\}$. 
Then $\X(a,b)$ is closed under direct summands and extensions.
Moreover, if $R\neq 0$, then $\X(a,b)$ is a proper subcategory of $\db(R)$.
Let $G$ be a nonzero object in $\db(R)$.
Then $G\in \X(\inf G,\sup G)$, and the closure properties above imply $\bigcup_{n\geq 0}{|G|}_{n}^{\db(R)}\subseteq \X(\inf G,\sup G)\subsetneq \db(R)$.
This proves the assertion.
\end{rem}
Following \cite{T2023}, we denote by $\c(R)$ the subcategory of $\mod R$ consisting of modules $M$ such that the inequality $\depth M_{\p} \ge \depth R_{\p}$ holds for all $\p \in \Spec R$.
If $R$ is Cohen--Macaulay, then $\c(R)$ coincides with the category $\CM(R)$ of maximal Cohen--Macaulay $R$-modules.
If $R$ is Gorenstein, then $\CM(R)$ is a Frobenius category, and its stable category $\underline{\CM}(R)$ admits a triangulated structure whose shift functor is given by the cosyzygy functor $\Omega^{-1}(-)$; see \cite[Chapter I]{Hap}.
Note that, if $\dim R=d<\infty$, then $\Omega^{d}(\mod R)$ is contained in $\c(R)$, where $\Omega^{d}(\mod R)$ denotes the subcategory of $\mod R$ consisting of modules $M$ for which there exists an exact sequence $0 \to M \to F_{d-1} \to \cdots \to F_0 \to N \to 0$ with each $F_i \in \proj R$ and $N \in \mod R$.
When $R$ is local, for $M \in \mod R$ we denote by $\Omega^n M$ the $n$-th syzygy of $M$ in its minimal free resolution.

The following proposition shows that, when $R$ is Gorenstein, the rank of $\ds(R)$ coincides with the (Dao--Takahashi) rank of $\CM(R)$, and that these ranks are finite under a mild assumption.
\begin{prop}\label{frob}
\begin{enumerate}[\rm(1)]
\item
Let $\mathcal{F}$ be a Frobenius category. 
Let $G$ be an object in $\mathcal{F}$ and $n$ a positive integer.
Then one has $\underline{\mathcal{F}}={|G|}_{n}^{\underline{\mathcal{F}}}$ if and only if $\mathcal{F}={|\{G\}\cup\proj\mathcal{F}|}_{n}^{\mathcal{F}}$.
\item
Let $R$ be a Gorenstein ring of finite Krull dimension.
Then the equalities
$\rank\ds(R)=\rank\underline{\CM}(R)=\rank\CM(R)$
hold.
If, moreover, $R$ is excellent, then these values are finite.
\end{enumerate}
\end{prop}
\begin{proof}
(1) The assertion follows from the standard construction of the stable category of a Frobenius category. For any conflation $0\to X\to Y\to Z\to 0$ in $\mathcal{F}$, the induced sequence $X\to Y\to Z\to\Omega^{-1}X$ in $\underline{\mathcal{F}}$ is an exact triangle. Conversely every exact triangle $X\to Y\to Z\to\Omega^{-1}X$ in $\underline{\mathcal{F}}$ is isomorphic to an exact triangle induced by a conflation $0\to X'\to Y'\to Z'\to 0$ in $\mathcal{F}$.
(2) The first equality follows from Buchweitz's triangle equivalence \cite[Theorem 4.4.1]{B}, and the second follows from (1).
By \cite[Corollary 3.12]{DLT}, there exist integers $s,n \ge 0$ and an $R$-module $G$ such that $\Omega^{s}(\mod R) \subseteq {|G|}_{n}^{R}$.
Taking $s$ large enough, we may assume that $s \ge d$, where $d=\dim R$.
Since $R$ is Gorenstein, we have $\Omega^{s}(\mod R)=\CM(R)$.
Thus $\CM(R) \subseteq {|G|}_{n}^{R}$.
Now let $M \in \CM(R)$. 
Then $M \in {|G|}_{n}^{R}$, and hence $\Omega^{d}M \in {|\Omega^{d}G\oplus R|}_{n}^{R}$.
As $R$ is Gorenstein, the $d$-th cosyzygy functor is quasi-inverse to the $d$-th syzygy functor on $\underline{\CM}(R)$, and therefore $M \in {|\Omega^{-d}\Omega^{d}G\oplus R|}_{n}^{R}$.
Since $\Omega^{-d}\Omega^{d}G$ is maximal Cohen--Macaulay, it follows that $\CM(R)={|\Omega^{-d}\Omega^{d}G\oplus R|}_{n}^{R}$.
Thus $\rank \CM(R)$ is finite.
\end{proof}
We next consider sufficient conditions for $\ds(R)$ to have finite rank without assuming that $R$ is Gorenstein. 
We first note that if ${\syz}^{n}(\mod R)$ has finite representation type for some integer $n\ge0$, then the same argument as in \cite[Proposition 2.5]{Chen2011} shows that $\ds(R)$ admits an additive generator; in particular, $\rank\ds(R)=0$.
The following lemma is easily proved by induction on $n$.
\begin{lem}\label{lem_shiftclosed}
Let $\S$ be a subcategory of a triangulated category $\T$. 
If $\S$ is closed under shifts, then ${\langle\S\rangle}_{n}={|\S|}_{n}$ for every integer $n\ge0$. 
In particular, for every subcategory $\X$ of $\T$, one has ${\langle\X\rangle}_{n}={|{\langle\X\rangle}_{1}|}_{n}$ for every integer $n\ge0$.
\end{lem}
The shift-closedness of the additive closure of a strong generator provides the following sufficient conditions for a triangulated category to have finite rank.
\begin{lem}\label{lem_addshift}
Let $\T$ be a triangulated category.
\begin{enumerate}[\rm(1)]
\item 
Suppose that $G_{+}$ is a strong generator of $\T$ and $G_{-}$ is an object of $\T$. 
Assume that there exist integers $m,b,N>0$ such that $G_{+}[m]\in\add G_{+}$ and $G_{+}[-i]\in{|G_{-}|}_{b}$ for every $i\ge N$. Then $\T$ has finite rank.
\item 
Suppose that $G_{+}$ and $G_{-}$ are strong generators of $\T$. If there exist integers $m,n>0$ such that $G_{+}[m]\in\add G_{+}$ and $G_{-}[-n]\in\add G_{-}$, then $\T$ has finite rank.
\end{enumerate}
\end{lem}
\begin{proof}
(1) Choose an integer $a>0$ such that $\T={\langle G_{+}\rangle}_{a}$, and put $H=\bigl(\bigoplus_{r=0}^{m-1}G_{+}[r]\bigr)\oplus\bigl(\bigoplus_{i=1}^{N-1}G_{+}[-i]\bigr)\oplus G_{-}$. 
Then we have $G_{+}[j]\in\add H$ for every $j\ge0$. 
We also have $G_{+}[-i]\in\add H$ for $1\le i<N$, whereas $G_{+}[-i]\in{|G_{-}|}_{b}\subseteq{|H|}_{b}$ for every $i\ge N$. 
Thus every shift of $G_{+}$ belongs to ${|H|}_{b}$, and therefore ${\langle G_{+}\rangle}_{1}\subseteq{|H|}_{b}$. 
By Lemma \ref{lem_shiftclosed}, $\T={\langle G_{+}\rangle}_{a}={|{\langle G_{+}\rangle}_{1}|}_{a}\subseteq{|{|H|}_{b}|}_{a}\subseteq{|H|}_{ab}$. 
Hence $\T$ has finite rank.
(2) Replacing $G_{+}$ by $\bigoplus_{i=0}^{m-1}G_{+}[i]$ and $G_{-}$ by $\bigoplus_{i=0}^{n-1}G_{-}[-i]$, respectively, we may assume that $m=n=1$. 
Choose an integer $b>0$ such that $\T={\langle G_{-}\rangle}_{b}$. 
Since $G_{+}\in{\langle G_{-}\rangle}_{b}$, there exist integers $c_{1},\ldots,c_{u}$ such that $G_{+}\in{|\bigoplus_{j=1}^{u}G_{-}[c_{j}]|}_{b}$. 
Set $N=\max\{1,c_{1},\ldots,c_{u}\}$. 
For every $i\ge N$, we have $G_{+}[-i]\in{|G_{-}|}_{b}$. 
Thus (1) applies and shows that $\T$ has finite rank.
\end{proof}
In view of \cite[Theorem 4.16]{DLMO}, we next consider conditions for $\ds(R)$ to have finite rank when $R$ is a local ring with an isolated singularity and $\ds(R)$ admits a strong generator. 
A stronger condition related to that in Proposition \ref{prop_isolsing}(2) is studied in \cite[Section 2]{CK}.
\begin{prop}\label{prop_isolsing}
Let $(R,\m,k)$ be a local ring with an isolated singularity.
\begin{enumerate}[\rm(1)]
\item
Assume that $\ds(R)$ admits a strong generator and $R$ is Golod. 
Then $\ds(R)$ has finite rank.
\item
Assume that $R$ is quasi-excellent, and there exist integers $a>b\ge0$ such that ${\syz}^{b}k\in\add{\syz}^{a}k$ (e.g., $R$ is Burch or $\m$ is quasi-decomposable). 
Then $\ds(R)$ has finite rank.
\end{enumerate}
\end{prop}
\begin{proof}
If $R$ is regular, then $\ds(R)=0$, and there is nothing to prove. 
Thus we may assume that $R$ is singular.
(1) By assumption, $\ds(R)$ admits a strong generator, while the isolated singularity assumption yields $\ds(R)=\thick(k)$ by \cite[Corollary 4.5]{T2014}. 
Hence $k$ is a strong generator of $\ds(R)$.
Set $e=\edim R$ and $t=\depth R$. 
Let $K$ be the Koszul complex on a minimal system of generators of $\m$, and set $h_{i}=\dim_{k}\H_{i}(K)$. 
By \cite[Theorem 1.1]{CDEKPU}, there is an isomorphism ${\syz}^{e+1}k\cong\bigoplus_{j=0}^{e-1}({\syz}^{j}k)^{\oplus h_{e-j}}$.
This implies that $k[-e-1]\in\add(\bigoplus_{j=0}^{e-1}k[-j])$ in $\ds(R)$. 
Put $G_{+}=k$ and $G_{-}=\bigoplus_{j=0}^{e}k[-j]$. 
Both $G_{+}$ and $G_{-}$ are strong generators of $\ds(R)$, and the preceding decomposition gives $G_{-}[-1]\in\add G_{-}$. 
On the other hand, $\H_{e-t}(K)\ne0$, so $h_{e-t}>0$. Hence ${\syz}^{t}k$ is a direct summand of ${\syz}^{e+1}k$, and therefore $k[-t]$ is a direct summand of $k[-e-1]$ in $\ds(R)$. 
Shifting by $e+1$, we obtain $k[e-t+1]\in\add k$, that is, $G_{+}[e-t+1]\in\add G_{+}$. 
By Lemma \ref{lem_addshift}(2), the assertion follows.
(2) By \cite[Corollary 3.12]{DLT}, there exist an $R$-module $G$ and integers $N,s>0$ such that ${\syz}^{N}(\mod R)\subseteq{|G|}_{s}$. 
In particular, $G$ is a strong generator of $\ds(R)$. 
Thus, $k$ is also a strong generator of $\ds(R)$. 
For every $i\ge N$, one has ${\syz}^{i}k \in{|G|}_{s}$, and hence $k[-i]\in{|G|}_{s}$ in $\ds(R)$. 
On the other hand, the assumption ${\syz}^{b}k\in\add{\syz}^{a}k$ gives $k[-b]\in\add k[-a]$ in $\ds(R)$, and hence $k[a-b]\in\add k$. 
By Lemma \ref{lem_addshift}(1), the assertion follows.
Finally, if $R$ is a Burch ring of depth $t$, then ${\syz}^{t}k$ is a direct summand of ${\syz}^{t+2}k$ by \cite[Proposition 5.10]{DKT}; if $\m$ is quasi-decomposable, then ${\syz}^{t+1}k$ is a direct summand of ${\syz}^{t+2}k$ by \cite[Proposition 4.3(3)]{T2026}.
Thus both classes satisfy the assumption of (2).
\end{proof}
We end this section by considering a weaker generation property for singularity categories.
\begin{rem}\label{entropy} 
Let $\T$ be a triangulated category, and let $\Delta(\T)$ be as in \cite[Definition 2.9]{AIT}. 
It follows from \cite[Lemma 2.11]{AIT} that if $\T=\bigcup_{t>0}{|G|}_{t}^{\T}$ for some $G\in\T$, then $\Delta(\T)$ is bounded. 
\end{rem} 
Contrary to Remark \ref{rank_db}, it often happens that $\ds(R)=\bigcup_{n>0}{|G|}_{n}$, as shown below. 
We first need a preparatory lemma, which is recorded in \cite[Remark 4.5(2)]{AIT}, but we include a proof for the convenience of the reader. 
\begin{lem}\label{cmextlem} 
Let $R$ be a Cohen--Macaulay local ring admitting a canonical module. 
If $\mod R$ admits an extension generator in the sense of \cite{DLMO}, then $\CM(R)=\bigcup_{t>0}{|H|}_{t}$ for some $H\in\CM(R)$. 
\end{lem} 
\begin{proof} 
Set $d=\dim R$. 
Passing to a sufficiently high syzygy, we may assume that $\syz^n\CM(R)\subseteq\bigcup_{t>0}{|G|}_{t}$ for some $n\ge1$ and some $G\in\mod R$. 
Let $\omega$ be a canonical module of $R$, and for $M\in\mod R$, set $M^{\dagger}=\Hom_R(M,\omega)$. 
Let $M\in\CM(R)$. There is an exact sequence $0\to N\to F_{d+n-1}\to\cdots\to F_0\to M^{\dagger}\to0$, where the $F_i$ are free modules and $N\in\CM(R)$. 
Dualizing by $\omega$ and using $\Ext_R^{>0}(M^{\dagger},\omega)=0$ and $M^{\dagger\dagger}\cong M$, we obtain an exact sequence $0\to M\to\omega^{\oplus b_0}\to\cdots\to\omega^{\oplus b_{d+n-1}}\to L\to0$, where $L\cong N^{\dagger}\in\CM(R)$. 
By \cite[Lemma 5.8]{DT2014}, we have $M\in{|\syz^{d+n}L\oplus W|}_{d+n+1}$, where $W=\bigoplus_{j=0}^{d+n-1}\syz^j\omega\in\CM(R)$. 
Since $\syz^{d}L\in\CM(R)$, we get $M\in{|\syz^{n+d}L\oplus W|}_{d+n+1}\subseteq\bigcup_{t>0}{|\syz^{d}G\oplus W|}_{t}$. 
As $M\in\CM(R)$ was arbitrary, we conclude that $\CM(R)\subseteq\bigcup_{t>0}{|\syz^{d}G\oplus W|}_{t}$, and the reverse inclusion is clear since $\syz^{d}G\oplus W\in\CM(R)$. 
\end{proof}
\begin{prop}\label{extdsg} 
Let $R$ be a Cohen--Macaulay local ring admitting a canonical module and locally Gorenstein on the punctured spectrum. 
If $\mod R$ admits an extension generator in the sense of \cite{DLMO}, then $\ds(R)=\bigcup_{n>0}{|H|}_{n}^{\ds(R)}$ for some $H\in\mod R$. 
\end{prop} 
\begin{proof} 
By Lemma \ref{cmextlem}, we have $\CM(R)=\bigcup_{t>0}{|H|}^{R}_{t}$ for some $H\in\CM(R)$. 
By \cite[Proposition 4.4(1)]{AIT}, there exists $N\in\CM(R)$ such that $H\in\bigcup_{t>0}{|N[-1]|}^{\ds(R)}_{t}$. 
Hence $H[1]\in\bigcup_{t>0}{|N|}^{\ds(R)}_{t}\subseteq\bigcup_{t>0}{|H|}^{\ds(R)}_{t}$, where the last inclusion follows from $N\in\CM(R)$. 
Thus $H[1]\in{|H|}^{\ds(R)}_{a}$ for some $a>0$, and hence $H[j]\in{|H|}^{\ds(R)}_{a^j}$ for every $j>0$.
Now let $X\in\ds(R)$. 
Then $X\cong M[i]$ in $\ds(R)$ for some $M\in\CM(R)$ and $i\in\mathbb Z$. 
If $i\le0$, then $M[i]\cong\syz^{-i}M$ in $\ds(R)$, and since $\syz^{-i}M\in\CM(R)$, we have $M[i]\in\bigcup_{t>0}{|H|}_{t}^{\ds(R)}$. 
If $i>0$, choose $b>0$ such that $M\in{|H|}_{b}^{R}$.
Then $M[i]\in{|H[i]|}_{b}^{\ds(R)}\subseteq{|H|}_{ba^i}^{\ds(R)}$. 
This proves the assertion. 
\end{proof} 
\begin{rem}
\begin{enumerate}[\rm(1)]
\item 
In view of Remark \ref{entropy}, Proposition \ref{extdsg} may be viewed as a refinement of \cite[Proposition 4.4(2)]{AIT}, as it gives the intermediate generation property $\ds(R)=\bigcup_{n>0}{|H|}_{n}^{\ds(R)}$.
\item 
Under the assumptions of Proposition \ref{extdsg}, the proof of \cite[Proposition 4.4(1)]{AIT} shows that there exists a positive integer $c$ such that $\CM(R)[1]\subseteq{|\CM(R)|}_{c}^{\ds(R)}$.
\end{enumerate}
\end{rem}
\section{Ranks of Verdier quotients}
In this section, we study the rank of $\d(R)$ in terms of its localization on the punctured spectrum. 
We first establish lower and upper bounds for the rank and then use them to prove the main result of this paper.
We begin by recalling some loci associated with a noetherian ring $R$.
We denote by $\Max(R)$ the set of maximal ideals of $R$, and by $\Min(R)$ the set of minimal prime ideals of $R$.
\begin{dfn}\label{def_locus}
Let $R$ be a noetherian ring.
\begin{enumerate}[\rm(1)]
\item
For a subset $W$ of $\Spec R$, we define the dimension of $W$, denoted by $\dim W$, to be the supremum of the integers $n$ such that there exists a chain $\p_{0}\subsetneq \p_{1}\subsetneq \cdots \subsetneq \p_{n}$ of prime ideals in $W$.
\item
A subset $\Phi$ of $\Spec R$ is {\em specialization-closed} if $\q\in\Phi$ for all prime ideals $\p\subseteq\q$ with $\p\in\Phi$.
If $\Phi$ is a specialization-closed subset of $\Spec R$, then one has $\dim\Phi=\sup\{\dim R/\p\mid\p\in\Phi\}$.
\item
We denote by $\Sing(R)$ the {\em singular locus} of $R$, that is, the set of prime ideals $\p$ such that $R_{\p}$ is not regular.
We say that $R$ is {\em J-1} if $\Sing(R)$ is a closed subset of $\Spec R$.
This terminology follows \cite{Mat.CA}.
We say that $R$ has {\em isolated singularities} if $\dim\Sing(R)\le 0$.
\item
Let $M$ be a finitely generated $R$-module.
We denote by $\NF(M)$ the {\em nonfree locus} of $M$, that is, the set of prime ideals $\p$ of $R$ such that $M_{\p}$ is not a free $R_{\p}$-module.
\end{enumerate}
\end{dfn}
\begin{rem}\label{rem_locus}
Let $R$ be a commutative noetherian ring.
\begin{enumerate}[\rm(1)]
\item
The subset $\Sing(R)$ is specialization-closed.
Any subset of $\Max(R)$ is specialization-closed.
\item
Recall that $\Spec R$ is finite if and only if $R$ is semilocal and $\dim R\le 1$.
Hence, if $R$ is semilocal, then $\Sing(R)$ is finite if and only if $R$ is J-1 and $\dim\Sing(R)\le 1$, and this is equivalent to saying that $\Sing(R)=\{\m_{1},\ldots,\m_{a},\p_{1},\ldots,\p_{b}\}$, where $\m_{i}\in\Max(R)$ and $\p_{j}\notin\Max(R)$.
\item
If $d=\dim R$ is finite, then one has $\NF(\Omega^{n}M)\subseteq\Sing(R)$ for all $n\ge d$ and all $M\in\mod R$.
\end{enumerate}
\end{rem}
We record the following observation on the localization of Verdier quotients associated with specialization-closed subsets of $\Spec R$.
\begin{rem}\label{rmk_vl}
Let $\Phi$ be a specialization-closed subset of $\Spec R$, and set $\T=\db(R)/\thick\{R,R/\p\mid \p\in\Phi\}$. 
If $\q\in\Spec R\setminus\Phi$, then the localization functor $(-)_{\q}:\db(R)\to\db(R_{\q})$ induces an exact dense functor $\T\to\ds(R_{\q})$. 
Indeed, there is no prime ideal $\p\in\Phi$ such that $\p\subseteq\q$, since $\Phi$ is specialization-closed.
Hence $(R/\p)_{\q}=0$ for every $\p\in\Phi$. 
On the other hand, the localization functor $(-)_{\q}\colon\db(R)\to\db(R_{\q})$ is exact and dense by \cite[Lemma 4.2(1)]{AT2015}.
Therefore it induces an exact dense functor $\T\to\db(R_{\q})/\thick\{R_{\q}\}=\ds(R_{\q})$.
\end{rem}
We give a lower bound for the rank of $\d(R)$ in terms of the ranks of singularity categories on the punctured spectrum.
For basic properties of $\ds(R)$ and $\db(R)$, we refer the reader to \cite[Lemma 2.4]{DT2015b} and \cite[Lemma 3.9]{MT}.
The following proposition generalizes \cite[Theorem 4.4]{MT}.
\begin{prop}\label{lower}
Let $R$ be a commutative noetherian ring.
Let $\Phi$ be a specialization-closed subset of $\Spec R$.
Set $\T=\db(R)/\thick\{R,R/\p\mid\p\in\Phi\}$.
Then the following hold.
\begin{enumerate}[\rm(1)]
\item
Let $\X$ be a subcategory of $\db(R)$ and $n\geq0$ an integer.
If $\T={|\X|}_{n}^{\T}$, then we have $\ds(R_{\q})={|\X_{\q}|}_{n}^{\ds(R_{\q})}$ for all $\q\in\Spec R\setminus\Phi$.
\item 
One has
\begin{align*}
\rank\T&\ge\sup\{\rank\ds(R_{\q})\mid\q\in\Spec R\setminus\Phi\} \\
&\ge\sup\{\rank\ds(R_{\q})\mid\q\in\Sing(R)\setminus\Phi\}.
\end{align*}
In particular, if $(R,\m,k)$ is local, then the inequality $\rank\d(R)\ge\sup\{\rank\ds(R_{\q})\mid\q\in\Sing(R)\setminus\{\m\}\}$ holds.
\end{enumerate}
\end{prop}
\begin{proof}
(1) Let $\q$ be a prime ideal in $\Spec R\setminus\Phi$.
By Remark \ref{rmk_vl}, the localization functor induces an exact dense functor $\T\to\ds(R_{\q})$.
Hence the equality $\T={|\X|}^{\T}_{n}$ yields $\ds(R_{\q})={|\X_{\q}|}^{\ds(R_{\q})}_{n}$.
The assertion in (2) immediately follows from (1).
\end{proof}
Now we relate generation of the Verdier quotient $\db(R)/\thick\{R,R/\p\mid \p\in\Phi\}$, where $\Phi$ is a subset of $\Spec R$, to bounds on the size of the singular locus of $R$.
\begin{prop}\label{sing}
Let $R$ be a commutative noetherian ring and $\Phi$ a subset of $\Spec R$.
Set $\T=\db(R)/\thick\{R,R/\p\mid\p\in\Phi\}$.
\begin{enumerate}[\rm(1)]
\item
Assume that $\T$ admits a classical generator (i.e., there exists an object $G\in\T$ such that $\T=\thick G$).
\begin{enumerate}[\rm(a)]
\item
If $R$ is semi-local, $\Phi$ is closed, and $\dim\Phi\le 1$, then $\Phi$ is finite.
\item
If $\Phi$ is finite, then $\Sing(R)$ is closed.
\item
If $\Phi$ is a nonempty specialization-closed subset and $R_{\q}$ has an isolated singularity for each $\q\in\Spec R\setminus\Phi$, then we have $\dim\Sing(R)\le \dim\Phi +1$.
\item
If $\Phi$ is finite and $R_{\q}$ has an isolated singularity for each $\q\in\Spec R\setminus\Phi$, then $\Sing(R)$ is finite.
\end{enumerate}
\item
Assume that $\Phi$ is specialization-closed, $R$ is locally Gorenstein on $\Spec R\setminus\Phi$, and $\T$ admits an additive generator.
\begin{enumerate}[\rm(a)]
\item
For each $\q\in\Spec R\setminus\Phi$, the ring $R_{\q}$ has an isolated singularity.
\item
If $\Phi$ is nonempty, then we have $\dim\Sing(R)\le \dim\Phi +1$.
\item
If $\Phi$ is finite, then $\Sing(R)$ is also finite.
\end{enumerate}
\end{enumerate}
\end{prop}
\begin{proof}
(1)(a) The assertion follows from Remark \ref{rem_locus}(2).
(1)(b) By \cite[Lemma 7.5(1)]{T2017}, we have $\db(R)=\thick\{G,R,R/\p\mid\p\in\Phi\}$. 
Hence the object $G\oplus (\oplus_{\p\in\Phi}R/\p)$ is a classical generator of $\ds(R)$. 
Therefore $\Sing(R)$ is closed by \cite[Lemma 2.9]{IT2019}.
(1)(c) If $\dim\Phi=\infty$, there is nothing to prove. 
Put $d=\dim\Phi$. 
Suppose, to the contrary, that $\dim\Sing(R)>d+1$. 
Then there exists a chain of prime ideals $\p_0\subsetneq\p_1\subsetneq\cdots\subsetneq\p_{d+2}$ with $\p_i\in\Sing(R)$ for all $0\le i\le d+2$. 
If $\p_1\in\Phi$, then the specialization-closedness of $\Phi$ implies $\p_i\in\Phi$ for all $1\le i\le d+2$.
Hence $\p_1\subsetneq\p_2\subsetneq\cdots\subsetneq\p_{d+2}$ is a chain in $\Phi$ of length $d+1$, which contradicts $d=\dim\Phi$. 
Thus $\p_1\in\Spec R\setminus\Phi$. 
Since $R_{\p_1}$ has an isolated singularity and $\p_0R_{\p_1}$ is a non-maximal prime ideal of $R_{\p_1}$, the localization $(R_{\p_1})_{\p_0R_{\p_1}}\cong R_{\p_0}$ is regular.
This contradicts $\p_0\in\Sing(R)$. 
Therefore $\dim\Sing(R)\le \dim\Phi+1$.
(1)(d) By (1)(b), the subset $\Sing(R)$ is closed. 
Write $\Sing(R)=\V(I)$ for some ideal $I$ of $R$. 
Let $\q\in\V(I)$ be a non-minimal element of $\V(I)$. 
Then there exists $\p\in\V(I)$ such that $\p\subsetneq\q$. If $\q\notin\Phi$, then $R_{\q}$ has an isolated singularity. 
Since $\p R_{\q}$ is a non-maximal prime ideal of $R_{\q}$, the localization $(R_{\q})_{\p R_{\q}}\cong R_{\p}$ is regular. 
This contradicts $\p\in\V(I)=\Sing(R)$. 
Hence $\q\in\Phi$. 
Thus, we have $\V(I)\subseteq \Min\V(I)\cup\Phi$. 
Since $\Min\V(I)$ and $\Phi$ are finite, $\Sing(R)=\V(I)$ is also finite.
(2)(a) Let $\q\in\Spec R\setminus\Phi$. 
Since $\T$ admits an additive generator, $R_{\q}$ has finite CM type by Propositions \ref{frob}(2) and \ref{lower}(1).
Hence $R_{\q}$ has an isolated singularity by the Auslander--Huneke--Leuschke--Wiegand theorem; see \cite[Theorem 6.12]{LW}.
The assertions in (2)(b) and (2)(c) follow by combining (1)(c)(d) and (2)(a).
\end{proof}
To obtain an upper bound for the rank of $\d(R)$, we need to control shifts of modules in $\d(R)$. 
We introduce the following condition for this purpose.
Let $R$ be a noetherian local ring.
We say that $R$ satisfies condition ($\mathrm{CS}$) if for every $M\in\c(R)$ and every $n\in\ZZ$, there exists $N\in\c(R)$ such that $M[n]\in\add_{\d(R)}N$ in $\d(R)$.
Since $\c(R)$ is closed under syzygies, $R$ satisfies ($\mathrm{CS}$) if and only if there exists a positive integer $a$ such that for every $M\in\c(R)$, there exists $N\in\c(R)$ such that $M[a]\in\add_{\d(R)}N$ in $\d(R)$.
The following gives sufficient conditions for $R$ to satisfy ($\mathrm{CS}$).
\begin{prop}\label{suff_cs}
Let $(R,\m)$ be a noetherian local ring which is locally Gorenstein on the punctured spectrum.
Assume that either $R$ is Cohen--Macaulay or $\Sing(R)$ is finite.
Then $R$ satisfies $(\mathrm{CS})$.
\end{prop}
\begin{proof}
We may assume that $R$ does not have an isolated singularity.
The assertion in the Cohen--Macaulay case follows from the proof of \cite[Remark 4.9]{MT}.
We assume that $\Sing(R)$ is finite, and set $\Sing(R)=\{\m,\p_{1},\ldots,\p_{n}\}$, $\Phi=\{\p_{1},\ldots,\p_{n}\}$, and $\X=\Omega^{d+1}(\mod R)$, where $d=\dim R$.
Note that $\add_{R_{\p}}\X_{\p}=\add_{R_{\p}}\Omega_{R_{\p}}^{d+1}(\mod R_{\p})$ for all $\p\in\Spec R$.
Let $M\in\c(R)$ and $\p\in\Phi$. 
Since $R_{\p}$ is Gorenstein, one has $\c(R_{\p})=\Omega_{R_{\p}}^{d+1}(\mod R_{\p})=\add_{R_{\p}}\X_{\p}$.
Hence $M_{\p}\in\add_{R_{\p}}\X_{\p}$ for all $\p\in\Phi$.
Thus, by \cite[Lemma 3.7(2)]{T2023}, there exists an exact sequence $0\to L\to M\oplus M'\to X\to 0$ such that $\NF(L)\subseteq\NF(M)$ and $\NF(L)\cap\Phi=\emptyset$.
Since $M\in\c(R)$, one has $\NF(L)\subseteq\Sing(R)$.
Hence $\NF(L)\subseteq\{\m\}$, and we have $M\oplus M'\cong X$ in $\d(R)$.
We write $X=\Omega^{d+1}X'\oplus R^{\oplus}$, where $X'\in\mod R$.
Set $N=\Omega^{d}X'$.
Then $N\in\c(R)$, and $M[1]$ is a direct summand of $X[1]\cong \Omega^{d}X'=N$ in $\d(R)$.
Therefore we conclude that $M[1]\in\add_{\d(R)}N$.
\end{proof}
The following lemma provides a sufficient condition for a given subcategory of $\d(R)$ to generate the whole category.
\begin{lem}\label{drgen}
Let $R$ be a noetherian local ring, and $\X$ a subcategory of $\d(R)$. 
Assume that $R$ satisfies $(\mathrm{CS})$ and there exists an integer $n\ge0$ such that $\Omega^{n}\c(R)\subseteq\add_{\d(R)}\X$.
Then one has $\d(R)=\add_{\d(R)}\X$.
\end{lem}
\begin{proof}
Let $X\in\d(R)$.
By \cite[Lemma 2.4(2)]{DT2015b}, there exist $M\in\c(R)$ and $l\in\ZZ$ such that $X\cong\Omega^{n}M[l]$ in $\d(R)$.
Since $R$ satisfies $(\mathrm{CS})$, there exists a module $N\in\c(R)$ such that $\Omega^{n}M[l+n]\in\add_{\d(R)}N$ in $\d(R)$.
By hypothesis, one has $\Omega^{n}N\in\add_{\d(R)}\X$.
Hence $X\cong(\Omega^{n}M[l+n])[-n]\in\add_{\d(R)}N[-n]=\add_{\d(R)}\Omega^{n}N\subseteq \add_{\d(R)}\X$.
Thus $\d(R)\subseteq\add_{\d(R)}\X$, and the reverse inclusion is clear.
Therefore we have $\d(R)=\add_{\d(R)}\X$.
\end{proof}
The following proposition shows that, for a ring $R$ satisfying $(\mathrm{CS})$, the category $\d(R)$ has finite rank under a mild assumption.
\begin{prop}\label{rankfin1}
Let $(R,\m)$ be a quasi-excellent local ring.
\begin{enumerate}[\rm(1)]
\item
If $R$ satisfies $(\mathrm{CS})$, then $\rank\d(R)<\infty$.
\item
Assume that $R$ is locally Gorenstein on the punctured spectrum and that either $R$ is Cohen--Macaulay or $\Sing(R)$ is finite.
Then $\rank\d(R)<\infty$.
\end{enumerate}
\end{prop}
\begin{proof}
(1) By \cite[Corollary 3.12]{DLT}, there exist integers $s,r\ge0$ and $G\in\mod R$ such that $\Omega^{s}(\mod R)\subseteq{|G|}_{r}^{R}$.
Applying the composite of the natural functors $\mod R\to\db(R)\to\d(R)$, one has $\Omega^{s}(\mod R)\subseteq{|G|}_{r}^{\d(R)}$.
Hence, by Lemma \ref{drgen}, we obtain $\d(R)={|G|}_{r}^{\d(R)}$.

(2) The assertion follows from Proposition \ref{suff_cs} and part (1).
\end{proof}
The next lemma establishes the compatibility between localization and the construction $|-|$.
\begin{lem}\label{lem_add}
Let $\X$ be a subcategory of $\mod R$.
Let $\p$ be a prime ideal of $R$ and $n$ a nonnegative integer.
Then one has $\add_{R_{\p}}\{({|\X|}_{n}^{R})_{\p}\}={|\X_{\p}|}_{n}^{R_{\p}}$.
\end{lem}
\begin{proof}
The assertion follows by induction on $n$, since the localization functor $(-)_{\p}:\mod R \to \mod R_{\p}$ is exact and dense, and commutes with taking $\Ext^{1}(-,-)$.
\end{proof}
The next lemma plays an important role in passing from local information to global information.
\begin{lem}\label{lem_lift}
Let $(R,\m)$ be a $d$-dimensional noetherian local ring.
Let $M$ be a finitely generated $R$-module, $\X$ a subcategory of $\mod R$, and $n$ a nonnegative integer.
Assume that $\Sing(R)$ is finite and that, for each $\p\in\Sing(R)\setminus\{\m\}$, one has $\Omega^{d}_{R_{\p}}M_{\p}\in{|\X_{\p}|}_{n}^{R_{\p}}$.
Then $\Omega^{s+d}M\in{|\Omega^{s}\X|}_{n}^{\d(R)}$ for all $s\ge0$.
\end{lem}
\begin{proof}
We may assume that $R$ does not have an isolated singularity.
Since $\Omega^{d+s}_{R_{\p}}M_{\p}\in{|\Omega_{R_{\p}}^{s}\X_{\p}|}_{n}^{R_{\p}}$, it is enough to prove the case where $s=0$.
Suppose that $n=0$.
Let $\p\neq\m$ be a prime ideal of $R$.
If $\p\in\Sing(R)$, then $\Omega_{R_{\p}}^{d}M_{\p}=0$ by hypothesis.
If $\p\notin\Sing(R)$, then $R_{\p}$ is a regular local ring of dimension less than $d$. 
Hence one has $\Omega_{R_{\p}}^{d}M_{\p}=0$.
Therefore, $\Omega^{d}M$ is locally free on the punctured spectrum, and this implies that $\Omega^{d}M=0$ in $\d(R)$.
Now assume that $n>0$.
Set $\Sing(R)=\{\m,\p_{1},\ldots,\p_{l}\}$, where $l\ge1$ and $\p_{i}\neq\m$ for all $1\le i\le l$.
Set $\Phi=\{\p_{1},\ldots,\p_{l}\}$ and $\Y={|\X\cup \{R\}|}_{n}^{R}$.
Then $\add_{R_{\p}}{\Y}_{\p}={|\{R_{\p}\}\cup  \X_{\p}|}_{n}^{R_{\p}}$ for all $\p\in\Spec R$ by Lemma \ref{lem_add}.
By hypothesis, one has $(\Omega^{d}M)_{\p}\in\add_{R_{\p}}{\Y}_{\p}$ for all $\p\in\Phi$.
Hence, by \cite[Lemma 3.7(2)]{T2023}, there exists an exact sequence $0\to L\to\Omega^{d}M\oplus N\to Y\to0$ with $Y\in\Y$ such that $\NF(L)\subseteq\NF(\Omega^{d}M)$ and $\NF(L)\cap\Phi=\emptyset$.
Since $\NF(\Omega^{d}M)\subseteq \Sing(R)=\{\m\}\cup\Phi$, one has $\NF(L)\subseteq\{\m\}$. 
This implies that $L=0$ in $\d(R)$.
Hence the above exact sequence yields $\Omega^{d}M\oplus N\cong Y$ in $\d(R)$, and $Y\in{|\X|}_{n}^{\d(R)}$.
Therefore, $\Omega^{d}M\in{|\X|}_{n}^{\d(R)}$.
\end{proof}
We are now ready to establish an upper bound for the rank of $\d(R)$.
\begin{thm}\label{upper}
Let $(R,\m)$ be a $d$-dimensional noetherian local ring.
Assume that $R$ satisfies $(\mathrm{CS})$, is J-1, and satisfies $\dim\Sing(R)\le1$.
Then
$\rank\d(R)\le\sup\{0,\sup_{\p\in\Sing(R)\setminus\{\m\}}\size\Omega^{d}(\mod R_{\p})\}$.
\end{thm}
\begin{proof}
We may assume that $\dim\Sing(R)=1$.
Indeed, if $\dim\Sing(R)\le0$, then $R$ has an isolated singularity, and the assertion is clear.
Set $\Sing(R)=\{\m,\p_{1},\ldots,\p_{l}\}$, where $\p_{i}\notin\Max(R)$ for all $1\le i\le l$.
We may assume that $n=\sup\{\size\Omega^{d}(\mod R_{\p})\mid\p\in\Sing(R)\setminus\{\m\}\}$ is finite.
Then, for each $1\le i\le l$, there exists $G_{i}\in\mod R_{\p_{i}}$ such that $\Omega^{d}(\mod R_{\p_{i}})\subseteq{|G_{i}|}^{R_{\p_{i}}}_{n+1}$.
We can take $Y_{i}\in\mod R$ such that $(Y_i)_{\p_i}\cong G_{i}$ for each $1\le i\le l$, and set $Y=\bigoplus_{i=1}^{l}Y_{i}$.
Then one has $\Omega^{d}(\mod R_{\p_{i}})\subseteq{|Y_{\p_{i}}|}^{R_{\p_{i}}}_{n+1}$ for all $1\le i\le l$.
Hence, by Lemma \ref{lem_lift}, we have $\Omega^{d}(\mod R)\subseteq{|Y|}_{n+1}^{\d(R)}$.
Therefore, by Lemma \ref{drgen}, one has $\d(R)={|Y|}_{n+1}^{\d(R)}$.
Thus we conclude that $\rank\d(R)\le n$.
\end{proof}
As an application of Theorem \ref{upper}, we characterize the rank of $\d(R)$ in terms of the ranks of the categories of maximal Cohen--Macaulay modules over the localizations of $R$ at prime ideals in the punctured spectrum.
\begin{cor}\label{equality}
Let $(R,\m)$ be a noetherian local ring which is locally Gorenstein on the punctured spectrum.
Assume that $R$ is J-1 and $\dim\Sing(R)\le1$.
Then the following equality holds:
\[
\rank\d(R)=\sup\{0,\sup_{\p\in\Sing(R)\setminus\{\m\}}\rank\CM(R_{\p})\}.
\]
\end{cor}
\begin{proof}
We may assume that $\dim\Sing(R)=1$.
By Remark \ref{rem_locus}(2), the singular locus of $R$ is finite. 
Hence Proposition \ref{suff_cs} shows that $R$ satisfies $(\mathrm{CS})$.
Set $d=\dim R$. 
For every $\p\in\Sing(R)\setminus\{\m\}$, the local ring $R_{\p}$ is Gorenstein and $d\ge\dim R_{\p}$, and hence $\Omega^{d}(\mod R_{\p})=\CM(R_{\p})$. 
Moreover, Proposition \ref{frob}(2) gives $\rank\ds(R_{\p})=\rank\CM(R_{\p})$.
Thus Proposition \ref{lower} and Theorem \ref{upper} yield
\begin{align*}
\sup\{\rank\CM(R_{\p})\mid\p\in\Sing(R)\setminus\{\m\}\}&\le\rank\d(R)\\
&\le\sup\{\size\CM(R_{\p})\mid\p\in\Sing(R)\setminus\{\m\}\}\\
&\le\sup\{\rank\CM(R_{\p})\mid\p\in\Sing(R)\setminus\{\m\}\}.
\end{align*}
Thus all the inequalities are equalities.
\end{proof}
We now give applications of Corollary \ref{equality} by constructing examples for which the rank of $\d(R)$ is equal to one.
The first example is a one-dimensional Cohen--Macaulay local ring.
\begin{ex}\label{eg_dim1}
Let $k$ be a field. 
Consider the ring
\[
R=k[\![x,y,z]\!]/\bigl((x^2,y^2)\cap(y,z^2)\bigr)
 =k[\![x,y,z]\!]/(y^2,x^{2}y,x^{2}z^{2}).
\]
The primary decomposition of the defining ideal shows that $\Spec R=\{\m,(x,y)R,(y,z)R\}$.
We have $\Ass(R)=\Min(R)$. 
Set $\p=(x,y)R$ and $\q=(y,z)R$.
Let $t=x-z$. 
Then $t$ is a nonzerodivisor on $R$, and $R/(t)\cong k[\![x,y]\!]/(y^2,x^{2}y,x^4)$.
Moreover, $\soc R/(t)=(x^3,xy)R/(t)$.
Hence $R$ is a one-dimensional Cohen--Macaulay non-Gorenstein local ring.
On the other hand, we have $R_{\p}\cong k[\![x,y,z]\!]_{(x,y)}/(x^2,y^2)$ and $R_{\q}\cong k[\![x,z]\!]_{(z)}/(z^2)$.
Thus $\Sing(R)=\{\m,\p,\q\}$.
Hence, by Corollary \ref{equality} and Example \ref{eg_rank1}, we obtain 
$\rank \d(R)=\sup\{\rank(\mod R_{\p}),\rank(\mod R_{\q})\}=1$.
\end{ex}
The next example is a two-dimensional local ring which is not Cohen--Macaulay.
\begin{ex}\label{eg_dim2}
Let $k$ be a field. 
Consider the ring
\[
R=k[\![x,y,z,w]\!]/\bigl((z,w)\cap(x,w)\cap(x^2,y,z^2)\bigr)
 =k[\![x,y,z,w]\!]/(x^2w,yw,z^2w,x^2z,xyz,xz^2).
\]
The above decomposition gives $\Ass(R)=\Min(R)=\{(z,w)R,(x,w)R,(x,y,z)R\}$.
In particular, $\dim R=2$ and $R$ is not equidimensional.
Set $t=y-w$.
Then $t$ is a nonzerodivisor on $R$, and $R/(t)\cong k[\![x,y,z]\!]/(x^2y,y^2,yz^2,x^2z,xyz,xz^2)$.
Moreover, $(x^2y,y^2,yz^2,x^2z,xyz,xz^2)=(x,y)\cap(y,z)\cap(x^2,y^2,z^2,xyz)$.
Hence $\m/(t)$ belongs to $\Ass(R/(t))$, and so $\depth R/(t)=0$.
Thus $R$ is a two-dimensional non-Cohen--Macaulay local ring.
Let $S=k[\![x,y,z,w]\!]$, and let $\p\in\Spec R\setminus\{\m\}$.
Let $P$ be the inverse image of $\p$ in $S$.
If $x\notin \p$, then $R_{\p}\cong S_P/(z,w)S_P$, which is regular.
If $y\notin \p$, then $R_{\p}\cong S_P/(w,xz)S_P$, which is a hypersurface of finite CM type by \cite[Theorem 4.18]{LW}.
If $z\notin \p$, then $R_{\p}\cong S_P/(x,w)S_P$, which is regular.
If $w\notin \p$, then $R_{\p}\cong S_P/(x^2,y,z^2)S_P$, which is a complete intersection.
Therefore $R$ is locally a complete intersection on the punctured spectrum.
Using the above descriptions, a direct computation gives $\Sing(R)=\{\m,(x,y,z)R,(x,z,w)R\}$.
Set $\p=(x,y,z)R$ and $\q=(x,z,w)R$.
Then $R_{\p}\cong k[\![x,z,w]\!]_{(x,z)}/(x^{2},z^{2})$, and $\rank(\mod R_{\p})=1$ by Example \ref{eg_rank1}.
Moreover, $R_{\q}\cong k[\![x,y,z]\!]_{(x,z)}/(xz)$, and $\rank\CM(R_{\q})=0$.
Hence, by Corollary \ref{equality}, we have $\rank\d(R)=1$.
\end{ex}
We give an example in which $\rank\d(R)$ is exactly two.
\begin{ex}\label{eg_rank2}
Let $k$ be a field.
Consider the ring \[
R=k[\![x,y,z,w]\!]/(x^{2}z,x^{2}w^{2},y^{3},y^{2}w^{2},z^{2})=k[\![x,y,z,w]\!]/\bigl((x^{2},y^{2},z^{2})\cap(y^{3},z,w^{2})\cap(x^{2},y^{3},z^{2},w^{2})\bigr).
\]
Set $\p=(x,y,z)R$ and $\q=(y,z,w)R$.
The above decomposition gives $\Spec R=\Ass(R)=\{\p,\q,\m\}$ and $\Min(R)=\{\p,\q\}$.
In particular, $\dim R=1$.
Since $\m\in\Ass(R)$, we have $\depth R=0$.
Thus $R$ is a one-dimensional non-Cohen--Macaulay local ring.
Let $S=k[\![x,y,z,w]\!]$.
Then one has $R_{\p}\cong S_{(x,y,z)}/(x^{2},y^{2},z^{2})$ and $R_{\q}\cong S_{(y,z,w)}/(y^{3},z,w^{2})\cong k[\![x,y,w]\!]_{(y,w)}/(y^{3},w^{2})$.
Therefore $R$ is locally a complete intersection on the punctured spectrum, and we have $\Sing(R)=\{\m,\p,\q\}$.
Moreover, $\rank(\mod R_{\p})=2$ and $\rank(\mod R_{\q})=1$ by Example \ref{eg_rank1}.
Hence, by Corollary \ref{equality}, we have $\rank\d(R)=\sup\{\rank(\mod R_{\p}),\rank(\mod R_{\q})\}=2$.
\end{ex}
As an application of Corollary \ref{equality}, we obtain another sufficient condition for $\d(R)$ to have finite rank.
\begin{cor}\label{rankfin2}
Let $R$ be a noetherian local ring which is locally excellent Gorenstein on the punctured spectrum.
Assume that $\Sing(R)$ is finite.
Then the rank of $\d(R)$ is finite.
\end{cor}
\begin{proof}
Combining Proposition \ref{frob}(2) with Corollary \ref{equality}, the assertion follows.
\end{proof}
We provide an example of Corollaries \ref{equality} and \ref{rankfin2}.
In the following example, the ring $R$ is not excellent, while $\d(R)$ is nontrivial and has finite rank.
\begin{ex}\label{eg_notexc}
There exists a one-dimensional local domain $S$ which is not excellent; see \cite{A1935} for instance.
Consider the ring $R=S[\![x,y]\!]/(x^{a},y^{b})$, where $a,b\ge 2$.
Then $R$ is a one-dimensional local ring which is not excellent. 
Set $\p=(x,y)R$.
Then one has
$R_{\p}\cong S[\![x,y]\!]_{(x,y)}/(x^{a},y^{b})
\cong Q(S)[x,y]/(x^{a},y^{b})$, where $Q(S)$ denotes the field of fractions of $S$.
Hence, $\Spec R=\Sing(R)=\{\p,\m\}$ and $R$ is locally excellent Gorenstein on the punctured spectrum.
Thus it follows from Corollary \ref{rankfin2} that $\rank\d(R)$ is finite.
In the present example, an argument similar to that in Example \ref{eg_rank1} shows that $1\le\rank(\mod R_{\p})\le \min\{a-1,b-1\}$.
Hence by Corollary \ref{equality} we have $1\le\rank\d(R)\le\min\{a-1,b-1\}$.
\end{ex}
We record here that the Rouquier dimension analogue of Corollary \ref{equality} also holds. 
We note that, when one considers Rouquier dimension, there is no need to control shifts. 
Thus condition $(\mathrm{CS})$ is not required, and consequently the assumption that $R$ is locally Gorenstein on the punctured spectrum can be omitted. 
For the definition of the operation $[-]^{R}_{n}$ in $\mod R$, we refer the reader to \cite{BSTT,DT2014}.
\begin{prop}\label{rdim}
Let $R$ be a commutative noetherian ring.
Let $\Phi$ be a specialization-closed subset of $\Spec R$, and set $\T=\db(R)/\thick\{R,R/\p\mid\p\in\Phi\}$.
Let $\X$ be a subcategory of $\db(R)$, let $\C$ be a subcategory of $\mod R$, and let $n\geq0$ be an integer.
\begin{enumerate}[\rm(1)]
\item
If $\T={\langle\X\rangle}_{n}^{\T}$, then we have  $\ds(R_{\q})={\langle\X_{\q}\rangle}_{n}^{\ds(R_{\q})}$ for all $\q\in\Spec R\setminus\Phi$.
\item
One has
\begin{align*}
\dim\T&\ge\sup\{\dim\ds(R_{\q})\mid\q\in\Spec R\setminus\Phi\} \\
&\ge\sup\{\dim\ds(R_{\q})\mid\q\in\Sing(R)\setminus\Phi\}.
\end{align*}
In particular, if $(R,\m,k)$ is local, then the inequality $\dim\d(R)\ge\sup\{\dim\ds(R_{\q})\mid\q\in\Sing(R)\setminus\{\m\}\}$ holds.
\item
Let $\p$ be a prime ideal of $R$.
Then one has $\add_{R_{\p}}\{({[\C]}_{n}^{R})_{\p}\}={[\C_{\p}]}_{n}^{R_{\p}}$.
\item
Let $(R,\m)$ be a $d$-dimensional noetherian local ring and $M$ a finitely generated $R$-module.
Let $a\ge d$ be an integer.
Assume that $\Sing(R)$ is finite and that, for each $\p\in\Sing(R)\setminus\{\m\}$, one has $\Omega^{a}_{R_{\p}}M_{\p}\in{[\C_{\p}]}_{n}^{R_{\p}}$.
Then $\Omega^{s+a}M\in{\langle\Omega^{s}\C\rangle}_{n}^{\d(R)}$ for all $s\ge0$.
\item
Assume that $(R,\m)$ is local and $\Sing(R)$ is finite.
If $\ds(R_{\p})={\langle\C_{\p}\rangle}_{n}^{\ds(R_{\p})}$ for all $\p\in\Sing(R)\setminus\{\m\}$, then one has $\d(R)={\langle\C\rangle}_{n}^{\d(R)}$.
\item
Assume that $(R,\m)$ is local, J-1, and $\dim\Sing(R)\le1$.
Then the following equality holds:
\[
\dim\d(R)=\sup\{0,\sup_{\p\in\Sing(R)\setminus\{\m\}}\dim\ds(R_{\p})\}.
\]
\end{enumerate}
\end{prop}
\begin{proof}
The assertions in (1)--(4) follow, respectively, by arguments analogous to those in Proposition \ref{lower} and Lemmas \ref{lem_add} and \ref{lem_lift}. 
We only point out the differences arising from the use of ${|-|}_{n}$, ${[-]}^{R}_{n}$, and ${\langle-\rangle}_{n}$.
For (1) and (2), for every $\q\in\Spec R\setminus\Phi$, the induced functor $\T\to\ds(R_{\q})$ is a triangle functor.
Hence, for every subcategory $\X$ of $\T$ and every $n\ge 0$, one has $({\langle\X\rangle}^{\T}_{n})_{\q}\subseteq{\langle\X_{\q}\rangle}^{\ds(R_{\q})}_{n}$.
For (3), localization commutes with taking syzygies up to free summands.
Therefore, the argument in Lemma \ref{lem_add} also applies to the operation ${[-]}^{R}_{n}$. 
Finally, for (4), let $\pi$ be the composition of the natural functors $\mod R\to\db(R)\to\d(R)$.
For a subcategory $\C$ of $\mod R$, the image of ${[\C]}^{R}_{n}$ under $\pi$ is contained in ${\langle\C\rangle}^{\d(R)}_{n}$; see \cite[Lemma 2.4]{DT2015b} or \cite[Lemma 3.9]{MT}.
(5) Let $X$ be an object of $\d(R)$. 
Then there exist $M \in \mod R$ and $l \in \mathbb{Z}$ such that $X \cong M[l]$ in $\d(R)$. 
Since $\Sing(R)$ is finite, there exists an integer $a \geq \dim R$ such that $\Omega_{R_{\p}}^{a} M_{\p} \in [\C_{\p}]_{n}^{R_{\p}}$ for all $\p \in \Sing(R) \setminus \{\m\}$. 
Here we use the following observation. 
The proof of \cite[Proposition 5.3]{BSTT} also shows that the object $G$ in that proposition can be replaced by an arbitrary subcategory $\C$ of a resolving subcategory $\X$.
Hence, by (4), we have $\Omega^{a}M \in {\langle \C \rangle}_{n}^{\d(R)}$.
Therefore $X \cong \Omega^{a}M[l+a] \in {\langle \C \rangle}_{n}^{\d(R)}$.
(6) We may assume that $\dim\Sing(R)=1$.
The lower bound follows immediately from (2). 
For the upper bound, it follows at once from (5), once one observes that a strong generator of the singularity category may be chosen to be a module, without increasing the number of cones required.
\end{proof}
\begin{rem}\label{rmk_cont}
Combining Propositions \ref{lower}, \ref{suff_cs} and Lemmas \ref{drgen}, \ref{lem_lift}, one obtains the following assertion.
Let $(R,\m)$ be a noetherian local ring. 
Assume that $\Sing(R)$ is finite and $R$ is Gorenstein on the punctured spectrum.
Then $R$ is of countable CM type on the punctured spectrum if and only if $\d(R)$ is the additive closure of a countable collection of objects.
However, if either of the equivalent conditions above holds, then the local ring $R_{\p}$ is of countable CM type and has an isolated singularity for every $\p \in \Sing(R) \setminus \{\m\}$.
Hence, from the viewpoint of the question of Huneke--Leuschke \cite{HL2003} (see also \cite{S2014}), this is almost the same as assuming finite CM type on the punctured spectrum.
\end{rem}
In view of Remark \ref{rmk_cont}, the following question naturally arises.
\begin{ques}
Does the equality in Proposition \ref{rdim}(6) still hold when $\dim\Sing(R)\ge 2$?
\end{ques}
\section{Finite syzygy type on the punctured spectrum}
In this section, we point out that one of the assumptions in \cite[Theorem 4.7]{MT}, namely the assumption that $R$ is locally Cohen--Macaulay on the punctured spectrum, can be removed.
We say that a local ring $R$ is of {\em finite syzygy type} if $\Omega^{n}(\mod R)\subseteq \add G$ for some $G\in\mod R$ and some $n\ge0$.
In this case, one has $n\ge\dim R$, and $R$ has an isolated singularity by \cite[Proposition 4.6]{DLMO} and \cite[Corollary 3.12]{DKLO}.
\begin{prop}\label{MT4.7}
Let $R$ be a noetherian local ring.
Assume that $R$ is J-1 and locally of finite syzygy type on the punctured spectrum.
Then $\d(R)$ admits an additive generator.
\end{prop}
\begin{proof}
We may assume that $R$ does not have an isolated singularity.
By hypothesis, $R_{\p}$ has an isolated singularity for every $\p\in\Spec R\setminus\{\m\}$.
Hence we have $\dim\Sing(R)=1$.
Set $\Sing(R)=\{\m,\p_{1},\ldots,\p_{n}\}$.
By hypothesis, there is an integer $a\ge\dim R$ such that for each $1\le i\le n$, there exists a module $Z_{i}\in\mod R_{\p_{i}}$ with $\Omega^{a}(\mod R_{\p_{i}})\subseteq\add_{R_{\p_{i}}}Z_{i}$.
The rest follows from the proof of \cite[Theorem 4.7]{MT}, replacing $d$ with $a$.
Thus $\d(R)$ admits an additive generator.
\end{proof}
We end this paper by giving an application of Proposition \ref{MT4.7}.
The following example is not covered by \cite[Theorem 4.7]{MT}.
\begin{ex}
Let $k$ be a field, and consider the ring $R=k[\![x,y,z]\!]/(x^2,xy)$.
It is easy to see that $\Sing(R)=\{\m,(x,y)R\}$.
Set $\p=(x,y)R$.
Then $R_{\p}\cong k[\![x,y,z]\!]_{(x,y)}/(x^2,xy)$ is not Cohen--Macaulay, but it is of finite syzygy type by \cite[Example 3.17(2)]{DKLO}.
Indeed, the argument in \cite[Example 3.17(2)]{DKLO} works for the ring $S/(x^2,xy)$, where $S$ is a two-dimensional regular local ring with a regular system of parameters $x,y$.
Hence, by Proposition \ref{MT4.7}, we conclude that $\d(R)$ admits an additive generator.
\end{ex}
\begin{ac}
The authors would like to thank Ryo Takahashi for his thoughtful questions and helpful discussions.
Souvik Dey was partly supported by the Charles University Research Center program No.UNCE/SCI/022 and a grant GA CR 23-05148S from the Czech Science Foundation.
Yuki Mifune was partly supported by Grant-in-Aid for JSPS Fellows 25KJ1386.
\end{ac}


\begin{thebibliography}{99}
\bibitem{AIT}
{\sc T. Araya; K.-I. Iima; R. Takahashi}, 
Vanishing of DHKK complexities for singularity categories and generation of syzygy modules, {\em Kyoto J. Math.} (to appear), {\tt arXiv:2310.15475}.
\bibitem{AT2015}
{\sc T. Aihara; R. Takahashi}, Generators and dimensions of derived categories of modules, {\em Comm. Algebra} {\bf 43} (2015), no. 11, 5003--5029.
\bibitem{A1935}
{\sc Y. Akizuki}, Einige Bemerkungen \"uber prim\"are Integrit\"atsbereiche mit Teilerkettensatz, {\em Proc. Phys.-Math. Soc. Japan} {\bf 17} (1935), 327--336.
\bibitem{Aoki}
{\sc K. Aoki}, Quasiexcellence implies strong generation, {\em J. Reine Angew. Math.}, {\bf 780} (2021), 133--138.
\bibitem{BSTT}
{\sc A. Bahlekeh; S. Salarian; R. Takahashi; Z. Toosi}, Spanier-Whitehead categories of resolving subcategories and comparison with singularity categories, {\em Algebr. Represent. Theory} {\bf 25} (2022), no. 3, 595--613.
\bibitem{BIKO2010}
{\sc P. A. Bergh; S. B. Iyengar; H. Krause; S. Oppermann}, Dimensions of triangulated categories via Koszul objects, {\em Math. Z.} {\bf 265} (2010), no. 4, 849--864.
\bibitem{BV}
{\sc A. Bondal; M. van den Bergh}, Generators and representability of functors in commutative and noncommutative geometry, {\em Mosc. Math. J.} {\bf 3} (2003), no. 1, 1--36, 258.
\bibitem{B}
{\sc R.-O. Buchweitz}, Maximal Cohen--Macaulay modules and Tate cohomology, With appendices and an introduction by L. L. Avramov, B. Briggs, S. B. Iyengar and J. C. Letz, Math. Surveys Monogr. {\bf 262}, {\em American Mathematical Society, Providence, RI}, 2021.
\bibitem{Chen2011}
{\sc X.-W. Chen}, The singularity category of an algebra with radical square zero, {\em Doc. Math.} {\bf 16} (2011), 921--936.
\bibitem{CDEKPU}
{\sc D. T. Cuong; H. Dao; D. Eisenbud; T. Kobayashi; C. Polini; B. Ulrich}, Syzygies of the residue field over Golod rings, preprint (2024), {\tt arXiv:2408.13425v3}.
\bibitem{CK}
{\sc D. T. Cuong; T. Kobayashi}, On direct summands of syzygies of the residue field of a local ring, preprint (2025), {\tt arXiv:2510.24220v1}.
\bibitem{DKT}
{\sc H. Dao; T. Kobayashi; R. Takahashi}, Burch ideals and Burch rings, {\em Algebra Number Theory} {\bf 14} (2020), no. 8, 2121--2150.
\bibitem{DT2014}
{\sc H. Dao; R. Takahashi}, The radius of a subcategory of modules, {\em Algebra Number Theory} {\bf 8} (2014), no. 1, 141--172.
\bibitem{DT2015a}
{\sc H. Dao; R. Takahashi}, The dimension of a subcategory of modules, {\em Forum Math Sigma} {\bf 3} (2015), e19, 31 pp.
\bibitem{DT2015b}
{\sc H. Dao; R. Takahashi}, Upper bounds for dimensions of singularity categories, {\em C. R. Math. Acad. Sci. Paris} {\bf 353} (2015), no. 4, 297--301.
\bibitem{DKLO}
{\sc S. Dey; K. Kimura; J. Liu; Y. Otake}, On local rings of finite syzygy representation type, preprint (2025), {\tt arXiv:2507.17097v2}.
\bibitem{DLT}
{\sc S. Dey; P. Lank; R. Takahashi}, Strong generation for module categories, {\em J. Pure Appl. Algebra} {\bf 229} (2025), no. 10, Paper No. 108070, 16pp.
\bibitem{DLMO}
{\sc S. Dey; J. Liu; Y. Mifune; Y. Otake}, Generation of singularity categories and infinite injective dimension locus via annihilation of cohomologies, {\em Canad. J. Math.}, published online 8 June 2026, doi:10.4153/S0008414X26102338.
\bibitem{Hap}
{\sc D. Happel}, Triangulated categories in the representation theory of finite-dimensional algebras, London Mathematical Society Lecture Note Series {\bf 119}, {\it Cambridge University Press, Cambridge}, 1988.
\bibitem{HL2003}
{\sc C. Huneke; G. J. Leuschke}, Local rings of countable Cohen--Macaulay type, {\em Proc. Amer. Math. Soc.} {\bf 131} (2003), no. 10, 3003--3007.
\bibitem{IT2016}
{\sc S. B. Iyengar; R. Takahashi}, Annihilation of cohomology and strong generation of module categories, {\em Int. Math. Res. Not. IMRN} (2016), no. 2, 499--535.
\bibitem{IT2019}
{\sc S. B. Iyengar; R. Takahashi}, Openness of the regular locus and generators for module categories, {\em Acta Math. Vietnam.} {\bf 44} (2019), no. 1, 207--212.
\bibitem{LW}
{\sc G. J. Leuschke; R. Wiegand}, Cohen--Macaulay representations, Math. Surveys Monogr. {\bf 181}, {\em American Mathematical Society, Providence, RI}, 2012.
\bibitem{Mat.CA}
{\sc H. Matsumura}, Commutative algebra, Second edition, Mathematics Lecture Note Series, 56, {\em Benjamin/Cummings
Publishing Co., Inc., Reading, Mass.}, 1980.
\bibitem{M8}
{\sc Y. Mifune}, Structure of modules stably annihilated by a fixed ideal, preprint (2026), {\tt arXiv:2508.16137v2}.
\bibitem{MT}
{\sc Y. Mifune; R. Takahashi}, On a Verdier quotient of a derived category of a local ring,
{\em J. Math. Soc. Japan} {\bf 77} (2025), no. 2, 563--579.
\bibitem{NT2020}
{\sc S. Nasseh and R. Takahashi}, Local rings with quasi-decomposable maximal ideal, {\em Math. Proc. Cambridge Philos. Soc.} {\bf 168} (2020), no. 2, 305--322.
\bibitem{Rou}
{\sc R. Rouquier}, Dimensions of triangulated categories, {\em J. K-Theory} {\bf 1} (2008), no. 2, 193--256.
\bibitem{S2014}
{\sc B. Stone}, Non-Gorenstein isolated singularities of graded countable Cohen--Macaulay type, in {\em Connections between algebra, combinatorics, and geometry}, Springer Proc. Math. Stat., {\bf 76}, Springer, New York, 2014, 299--317.
\bibitem{T2014}
{\sc R. Takahashi}, Reconstruction from Koszul homology and applications to module and derived categories, {\em Pacific J. Math.} {\bf 268} (2014), no. 1, 231--248.
\bibitem{T2017}
{\sc R. Takahashi}, Thick subcategories over isolated singularities, {\em Pacific J. Math.} {\bf 291} (2017), no. 1, 183--211.
\bibitem{T2023}
{\sc R. Takahashi}, Dominant local rings and subcategory classification, {\em Int. Math. Res. Not. IMRN} (2023), no. 9, 7259--7318.
\bibitem{T2023c}
{\sc R. Takahashi}, Remarks on complexities and entropies for singularity categories, {\em C. R. Math. Acad. Sci. Paris} {\bf 361} (2023), 1611--1623.
\bibitem{T2026}
{\sc R. Takahashi}, Uniformly dominant local rings and Orlov spectra of singularity categories, {\em Math. Z.} {\bf 312} (2026), no. 4, Paper No. 117, 27pp.
\end{thebibliography}
\end{document}